\documentclass[12pt,draft]{amsart}
\usepackage{enumerate}
\usepackage{amsmath, amssymb, epsfig, graphicx, url}
\usepackage{color,todonotes}
  \usepackage{pbox}
  
\newtheorem{thm}{Theorem}[section]
\newtheorem{lem}[thm]{Lemma}
\newtheorem{cor}[thm]{Corollary}

\newtheorem{exam}[thm]{Example}
\newtheorem{defn}[thm]{Definition}
\newtheorem{problem}{Problem}

\newtheorem{unnumber}{}

\newtheorem{prepf}[unnumber]{Proof}
\newenvironment{pf}{\prepf\rm}{\endprepf}

\newcommand{\Ker}{\mathop{\mathrm{ker}}\nolimits}

\newcommand{\twist}[1]{^#1\kern-0.15em}% for twisted groups of Lie type

\newcommand{\pxl}{\mathop{\mathrm{PXL}}}

\newcommand{\Aut}{\mbox{\rm Aut\,}}

\newcommand{\rank}{\operatorname{rank}}

\newcommand{\trans}{\mathcal{T}_{n}}
\newcommand{\sym}{\mathcal{S}_{n}}

\newcommand{\agl}{\mbox{\rm AGL}}
\newcommand{\psl}{\mbox{\rm PSL}}
\newcommand{\pgl}{\mbox{\rm PGL}}

\newcommand{\pgaml}{\mbox{\rm P}\Gamma {\rm L}}
\newcommand{\agaml}{\mbox{\rm A}\Gamma {\rm L}}
\newcommand{\psigl}{\mbox{\rm P}\Sigma {\rm L}}

\newcommand{\GF}{\mbox{\rm GF}}

\newtheorem{preex}{Example}[section]
\newenvironment{example}{\preex\rm}{\endpreex}

\begin{document}

%%%%%%%%%
%%%%%%%%%

%\begin{frontmatter}

\title[Orbits of $k$-homogenous groups]{Orbits of Primitive $k$-Homogenous Groups on $(n-k)$-Partitions with Applications to Semigroups}
\author{Jo\~{a}o Ara\'{u}jo}
\author{Peter J. Cameron}

\address[Ara\'{u}jo]
{Universidade Aberta
and
CEMAT-CI\^{E}NCIAS \\ 
Departamento de Matem\'{a}tica, Faculdade de Ci\^{e}ncias\\
 Universidade de Lisboa, 1749-016, Lisboa, Portugal
}
\email{\url{jaraujo@ptmat.fc.ul.pt}}

\address[Cameron]{School of Mathematics and Statistics\\
University of St Andrews\\St Andrews, Fife KY16 9SS\\U.K.}
\email{\url{pjc20@st-andrews.ac.uk}}

%%%%%%%%%

\begin{abstract}{
Let $X$ be a finite set such that $|X|=n$, and let $k< n/2$. A group is $k$-homogeneous if it has only one orbit on the sets of size $k$. The aim of this paper is to prove some general results on permutation groups and then apply them to transformation semigroups. On groups we find the minimum number of permutations needed to generate $k$-homogeneous groups (for $k\ge 1$); in particular we show that $2$-homogeneous groups are $2$-generated. We also describe the orbits of $k$-homogenous groups on partitions  with $n-k$ parts, classify the $3$-homogeneous groups $G$ whose orbits on $(n-3)$-partitions are invariant under the normalizer of $G$ in  $S_n$, and describe the normalizers of $2$-homogeneous groups in the symmetric group.
Then these results are applied to extract information about transformation semigroups with given group of units, namely to prove results on their automorphisms and on the minimum number of generators. The paper finishes with some problems on permutation groups, transformation semigroups and computational algebra.  }
\end{abstract}

\medskip

\maketitle

\noindent{\em Date:} 17 December 2015
{\em Corresponding author: Jo\~{a}o Ara\'{u}jo, jjaraujo@fc.ul.pt}

%%%%%%%%%

%\begin{keyword}
%idempotent-generated semigroup  finite permutation group.
%\end{keyword}

%\end{frontmatter}

%%%%%%%%%
%%%%%%%%%

%\newpage

\section{Introduction}
Let $T_n$ be the full transformation monoid on $n$ points and let $S_n$ be its group of units, the symmetric group. The rank of a transformation $t\in T_n$ is the size of its image, and is denoted by $\rank(t)$. We say that a partition $P=(A_1,\ldots,A_k)$, where the $A_i$s {are in a sequence of non-increasing  sizes},  is of type $(|A_1|,\ldots, |A_k|)$. Given two sets $A,B\subseteq T_n$, we denote by $\langle A,B\rangle$ the semigroup generated by $A\cup B$; in case $B=\{t\}$, we will abuse notation writing $\langle A,t\rangle$, rather than $\langle A,\{t\}\rangle$.  

In what group theory concerns, this paper investigates  the minimal generating sets of some permutation groups $G\le S_n$ and their orbits on $(n-k)$-partitions (for $k\ge n/2$). For example, in the particular case of  $k=1$, the general problem we are considering reads as  the study of the orbits of $G$ on $(n-1)$-partitions, that is, the study of the orbits of $G$ on $2$-sets (a slight generalization of the key concept of \emph{orbitals}).  The main results for groups are of the following form (with $m$ and $m'$ appearing in several tables): 

\begin{thm}	Let $k\le \frac{n}{2}$ and let 
 $G\le S_n$ be a primitive $k$-homogenous group. Then,
 \begin{itemize}
 \item $G$ has $m$ orbits on the set of $(n-k)$-partitions; 
 \item the smallest number of elements needed to generate $G$ is $m'$. 
 \end{itemize}  
\end{thm}

{ A consequence of the previous result is that $2$-homogeneous groups are $2$-generated. That this result seems to have been unnoticed in the past, is perhaps the explanation for the fact that a very optimized algebra system such as GAP \cite{GAP} provides sets of generators for the $2$-transitive groups that in about $2/3$ of the cases have more than $2$ elements. 

The next main group theory result,  Theorem \ref{main3}, builds upon the previous and  provides a list  of 
$3$-homogeneous groups whose orbits on a given $(n-3)$-partition coincide
with those of their normalizers. (The list is complete except for one unresolved family.)}

{In addition to dealing with natural questions on permutation groups, these results were crucial  in order to generalize some  semigroup theory results,  as shown below in the sample Theorems \ref{sample2} and \ref{sample}, that we now introduce. 
%Given a set $A\subseteq T_n$ and a group $G\le S_n$,  we denote by $\langle A,G\rangle$ the semigroup generated by $A$  and $G$. }

In his very influential and cited paper \cite{mcalister} McAlister proved the following. 

\begin{thm}\cite{mcalister}
Let $G\le S_n$ and  $t\in T_n$ be any map of rank $n-1$.  Then $\langle G,t\rangle$ generates all rank $n-1$  transformations in $T_n$  if and only if the group $G$ has only one orbit on the $(n-1)$-partitions of $\{1,\ldots ,n\}$.
\end{thm}

Another important result is due to Levi.

\begin{thm}\cite{Le96}
Let $A_n \le G \le S_n$ and let $t \in T_n\setminus S_n$. Then the automorphism group of $\langle G,t\rangle$ is isomorphic to $S_n$.  
\end{thm}
 
The aim of this paper is to use the group theory results referred to above to generalize these results as follows.

 \begin{thm}\label{sample2}
Let $t$ be a singular map in $T_n$, and suppose that $t$ has kernel type $(l_1,\ldots,l_k)$, with $k\ge n/2$; let $G$ be a group having only one orbit in the partitions of that type. Let $S=\langle t,G\rangle\setminus G$. Then
\begin{enumerate}
\item the automorphisms of $\langle a,G\rangle$ are those induced under conjugation by the elements of the normalizer of $G$ in $S_n$, $$\mbox{Aut}(\langle t,G\rangle)\cong\mbox{N}_{S_n}(G);$$ 
 {
\item if $k\le n-2$, then $\langle G,t\rangle$ is generated by $3$ elements;
\item  let $A$ be a set of rank $k$ maps such that $\langle A,G\rangle$ generates all maps of rank at most $k$ and $A$ has minimum size among the sets with that property. Then $|A|$ is given in Table \ref{t:gen}.} 
\end{enumerate}
\begin{table}[htbp]
\[
\begin{array}{|c|c|c|}		
\hline
\hbox{rank} & \hbox{partition type} & |A|\\
\hline \hline 
n-1 & (2,1,\ldots,1) & 1 \\ 
\hline
n-2 & (2,2,1,\ldots,1) & 2 \\ 
    & (3,1,\ldots,1) & O(n) \\
\hline 
n-3 & (4,1,\ldots,1) & 144 \\
    & (3,2,1,\ldots,1) & 5  \\
    & (2,2,2,1,\ldots,1) & 3 \\ 
\hline
n-4 & (5,1,\ldots,1) & 15 \\
    & \hbox{other} & 5 \\
\hline    
k\ (n/2\le k\le n-5) & \hbox{any} & p(k) \\
\hline 
\end{array}
\]
\caption{\label{t:gen}Generating all maps of rank $k$}
\end{table}
\end{thm}
{Table  \ref{t:gen} should be read as follows: if a group $G$ has, for example, one orbit on partitions of type $(3,1,\ldots,1)$, then we need a set $A$ with $O(n)$  maps  of rank $n-2$  so that the semigroup $\langle G,A\rangle$ generates all maps of rank no larger than $n-2$. }
 
%The previous theorem generalizes McAlister's and Levi's above, and in turn is generalized by the following. 

In the previous theorem we require the group of units to be transistive on the kernel type of the singular map $t$. In the next theorem this condition is replaced by the weaker requirement of $G$ to be $|Xt|$-homogeneous. 
\begin{thm}\label{sample}
Let $G$ be a primitive group with just one orbit on $(n-k)$-sets, where
$1\le k\le n/2$. Let $t\in T_n$ be a rank $(n-k)$ map. Then
\begin{enumerate}
%\item\label{i1c} $\langle G,t\rangle\setminus G$ and $\langle g^{-1}tg \mid g\in G\rangle$ have the same idempotents;
%\item $\Aut(\langle G,t\rangle)\cong\{f\in S_n:f^{-1}\langle G,a\rangle f
%=\langle G,a\rangle\}$.
\item $\Aut(\langle G,t\rangle)\cong N_{S_n}(\langle G,t\rangle)$.
\item For $k\ge3$, the list of $3$-homogeneous groups that satisfy $$\Aut(\langle G,t\rangle)\cong N_{S_n}(G)$$ is the following: 
\begin{itemize}\itemsep0pt
\item $G=N_{S_n}(G)$, that is, 
\begin{itemize}\itemsep0pt
\item[(i)] $S_n$.
\item[(ii)] $\pgaml(2,q)$ for $k=3$.
\item[(iii)] $\agl(d,2)$ for $k=3$.
\item[(iv)] $\agaml(1,8)$,
$M_{11}$ ($k=4$), $M_{11}$ (degree~$12$, $k=3$), $M_{12}$ ($k=5$), $2^4:A_7$, $M_{22}:2$ ($k=3$),
$M_{23}$ ($k=4$), $M_{24}$ ($k=5$), and $\agaml(1,32)$ ($k=4$).
\end{itemize}
\item $G=A_n$;
\item $G=\agl(1,8)$, $\pgl(2,8)$, $\pgl(2,9)$, $M_{10}$, $\psl(2,11)$,
$M_{22}$, $\pxl(2,25)$, or $\pxl(2,49)$, with $k=3$ and $\lambda=(4,1,\ldots,1)$.
\end{itemize}
The list is complete with the possible exception of the groups $\pxl(2,q)$ for $q\ge169$. 
\item Let $A\subseteq T_n$ be a set of rank~$n-k$ maps such that $\langle G,A\rangle$ generates all
maps of rank at most $n-k$, and suppose $A$ has minimum size among the sets with that property. Then the size of   $A$ is bounded by the values in  Table~\ref{t:gens3b}. 
\end{enumerate}
\end{thm}

{
\begin{table}[htbp]
\[\begin{array}{|c|c|c|c|}
\hline
\pbox{5cm}{\mbox{\tiny{Rank}} \\ \text{\tiny{$n-k$}} }& { |A|} &\pbox{5cm}{\   \\ {\tiny Sample $k$-homogeneous groups}\\  {\tiny  attaining the bound for $|A|$} \\ }&\pbox{5cm}{{\tiny Minimum number of} \\ {\tiny generators for a primitive} \\ {\tiny $k$-homogeneous group}  }\\
\hline\hline
\text{\tiny{$n-1$}} & \frac{(n-1)}{2} &C_p,D_{p} \ (\mbox{{\tiny $n$ odd prime}})  & \frac{C\log n}{\sqrt{\log \log n}} \\
\text{\tiny{$n-2$}} & O(n^2) &\mbox{Example \ref{expl2.1}} &2\\
\text{\tiny{$n-3$}} & O(n^3) &\mathrm{PSL}(2,q),\mathrm{P}\Gamma\mathrm{L}(2,q) &2\\
\text{\tiny{$n-4$}} & 12160 &\mathrm{P}\Gamma\mathrm{L}(2,32)  \ (\mbox{{\tiny $n=33$}}) &2\\
\text{\tiny{$n-5$}} & 77&M_{24} \ (\mbox{{\tiny $n=24$}})&2\\
\text{\tiny{$n-k$ $(k\ge5)$}} & p(k)&S_n,A_n &2 \\
\hline
\end{array}\]
\caption{\label{t:gens3b} { Number of rank $n-k$ maps needed to together with a $k$-homogeneous group $G$ generate all the maps of rank not larger than  $n-k$.}}
\end{table}

As said, the two results above are just  sample theorems. For more detailed results we refer the reader to the sections below.}

In what semigroup theory concerns, this paper belongs to the general area of investigating how the  recent results on group theory, chiefly the classification of finite simple groups, can help the study of semigroups. (For other papers on this line of research, see for example \cite{andre,ABC,ABCRS,ArBeMiSc,circulant2,ArCa13,ArCa14,ArCa12,acmn,ArCaSt15,ArDoKo,ArMiSc,ben,Ata15,Le99,lmm,lm,mcalister,neu,symo} and the references therein.)
The typical object in this field is a semigroup generated by a set of non-invertible transformations $A\subseteq T_n\setminus S_n$ and a group of permutations $G$ contained in $S_n$. In this paper we are mainly concerned with the description of  automorphisms and minimal generating sets, for semigroups having special given group of units. 

%Here we will consider two types of groups: those that have only one orbit on the kernel type of some transformation of rank $n-k$ (for $k\le n/2$); and those that have only one orbit on the image of rank $n-k$ transformations. One typical result is the following. 
%
%\begin{thm}\label{main:sem}
%\begin{enumerate}
%\item Given any rank $n-3$ transformation $f\in T_n$ we give the list of $3$-homogeneous groups $G$ such that $\Aut(\langle G,f\rangle)\cong N_{S_n}(G)$.
%\item Let $G$ be a $3$-homogeneous group and let $A\subseteq T_n$ be a set of rank~$n-k$ maps such that $\langle G,A\rangle$ generates all
%maps of rank at most $n-3$. Then the size of $A$ is bounded by $O(n^3)$.
%\end{enumerate}
%\end{thm}

If $S$ is a semigroup and $U$ is a subset of $S$, then we say that {\em $U$
generates $S$} if every element of $S$ is expressible as a product of the
elements of $U$. The {\em rank} of a semigroup $S$, denoted by $\rank S$, is the
least number of elements in $S$ needed to generate $S$.  It is well-known that a
finite full transformation semigroup, on at least 3 points, has rank~3, while a
finite full partial transformation semigroup, on at least 3 points, has rank~4
(see \cite[Exercises~1.9.7 and 1.9.13]{Ho95}). The problem of determining the
minimum number of generators of a semigroup is classical, and has been studied
extensively; see, for example, \cite{ar2002,ArBeMiSc,ArMiSi,1,8,10,Kon,15,16} and the references therein. Given the importance of idempotent generated semigroups illustrated by the Erdos/Howie famous twin results (see \cite{erdos,howie} and also \cite{ar2000,erdos2}) the related notion of 
{\em idempotent rank} appeared as natural and has also been widely investigated; the same can be said about the concepts of   {\em relative rank} and {\em nilpotent rank}; see \cite{0,2,mitchell1,6,9,8.5,mitchell2,14}.
One of the goals of this paper is to contribute to this line of research.

Another classic topic in semigroup theory is the description of the automorphisms of semigroups. After the pioneer work of Schreier
\cite{Sc36} and Mal'cev \cite{Ma52}, proving that the group of automorphisms
of $T_n$ is isomorphic to $S_n$, a long sequence of new results followed (for example,  \cite{circulant2,abmn,ArDoKo,ArFeJe,ArKi,ArKo,ArKo1,arko,arko2,arko20,Le85,Le87,Le96,Li53,Ma67,Sullivan,Su61,symo,yang} and the references therein). In addition to the general interest  of studying automorphisms of mathematical structures, the description of automorphisms of semigroups turned out to be a key ingredient in Plotkin's \emph{universal algebraic geometry} \cite{plotkin} and \cite{belov2,belov,berzins,formanek,Lipyanski,MaPl07,MaPl04,Mas03,plzh06,plzh07}.
%Similar results have been obtained for the semigroup of partial transformations 
%by \v{S}utov \cite{Su61} and Magill \cite{Ma67},
%and for the symmetric inverse semigroup of all partial one-to-one
%transformations on $X$ by Liber \cite{Li53}. Sullivan
%\cite{Sullivan} and Levi \cite{Le85}, \cite{Le87} generalized the above results
%to the class of semigroups $\langle S_n,a\rangle$, where $a\in T_n$. In \cite{ArKo},
%the authors determined the automorphism group of the semigroup $T(X,\rho,R)$ of all $a\in T_n$
%that preserve both an equivalence relation $\rho$ on $\{1,\ldots,n\}$ and a transversal $R$. 
Here, we use the impressive progresses made in the theory of permutation groups during the last couple of decades, to contribute to this line of research by finding the automorphisms of semigroups with given group of units. 

{In Section \ref{groups} we prove the main theorems about the minimum number of generators of primitive groups, and we also give estimates on the number of orbits of primitive groups on $(n-k)$-partitions, for $k\ge n/2$. In Section \ref{orbs and norms} we tackle the problem of independent interest of classifying the permutation groups in which all
orbits on $(n-k)$-partitions are invariant under the normalizer. 
%Section \ref{semigroups} contains some technical results and notation. 
In Section \ref{sem1} we apply the results proved in the previous sections to describe automorphisms and ranks of semigroups in which its group of units has just one orbit on the kernel type of $t$. In Section \ref{sem2} we consider similar problems, but for semigroups whose group of units  has just one orbit on the image of $t$.  
{Section~\ref{norm} contains some comments on the normalizers of
$2$-homogeneous or primitive groups.}
The paper ends with a section of open problems.}

%Since unit elements can only be generated by unit elements, any generating set of a semigroup must include a generating set for the group of units. Thus, for  semigroups having non-trivial group of units, it is necessary to investigate the semigroup generated by    

\section{Group theory}\label{groups}

The aim of this section is to prove all the results of this form.

\begin{thm}	Let $k\le n/2$ and let 
 $G\le S_n$ be a primitive $k$-homogenous group. Then,
 \begin{itemize}
 \item $G$ has $m$ orbits on the set of $(n-k)$-partitions; 
 \item the smallest number of elements needed to generate $G$ is $m'$. 
 \end{itemize}  
\end{thm}

The case in which least can be said is the case of $k=1$. The \emph{rank} $r(G)$ of a
transitive permutation group $G$ (acting on $\{1,\ldots,n\}$) is the number of
$G$-orbits on ordered pairs from $\{1,\ldots,n\}$. To handle the case of $k=1$, we
need a slightly different parameter, the number $n_2(G)$ of $G$-orbits on the
set of $2$-subsets of $\{1,\ldots,n\}$. Clearly $(r(G)-1)/2\le n_2(G)\le r(G)-1$;
the lower bound holds when $G$ has odd order (since then no pair of points can
be interchanged by an element of $G$), and the upper bound when all the orbitals
of $G$ are self-paired. Note that $r(G)\le n$, with equality if and only if
$G$ is regular. In particular, a primitive group $G$ has $r(G)=n$ if and only
if $n$ is prime and $G$ is cyclic of order $n$. We thus see that $n_2(G)\le n-1$
for transitive groups $G$; equality is realised for an elementary abelian $2$-group acting regularly, but for primitive groups of degree greater than $2$ we have $n_2(G)\le(n-1)/2$, with equality only for the cyclic and dihedral groups of odd prime degree.

\begin{thm}	Let 
 $G\le S_n$ be a $1$-homogeneous (that is, transitive) permutation group. Then
 \begin{itemize}
 \item $G$ has $n_2(G)$ orbits on the set of $(n-1)$-partitions; 
 \item the smallest number of elements needed to generate $G$ is at most
 \[
 \begin{array}{ll}
\frac{C n}{\sqrt{ \log n}}&\mbox{if $G$ is transitive} 	\\ \\
\frac{C\log n}{\sqrt{\log \log n}}&\mbox{if $G$ is primitive,} 	
 \end{array}
 \]where $C$ is a universal constant.
 \end{itemize}  
\end{thm}
\begin{proof}
The $(n-1)$-partitions all have one part with two elements and all the other parts are singletons. Therefore the group has as many orbits on the $(n-1)$-partitions as orbits on   the set of $2$-sets.

McIver and Neumann \cite{McIN} showed that every subgroup of $S_n$ can be generated by
$\lfloor n/2\rfloor$ elements if $n\neq 3$, and by $2$ if $n=3$. This bound is
best possible for arbitrary subgroups, but for transitive or primitive subgroups it has been improved in \cite{Lu1,Lu2} to the statements in the theorem.\end{proof}

Now we are going to prove that 
%turn our attention to the number of generators of a $2$-homogenous group. 
%At the other extreme are results for large $k$, meaning $k>3$ with finitely
%many exceptions. In this case, the minimum number of generators is $2$,
%and the number of orbits on $(n-k)$-partitions is bounded by a function of $k$. 
the minimum number of generators of any $2$-homogeneous finite group
is~$2$. (It is worth observing that we could not find this observation  in the literature; we are grateful to Colva Roney-Dougal and Andrea Lucchini for independently confirming it.) The proof uses the following result proved by  Lucchini and
Menegazzo~\cite{lume}. Here $d(G)$ denotes the least  number of elements of $G$ needed to generate the whole $G$. 

\begin{thm}[\cite{lume}]
Let $G$ be a non-cyclic finite group having a unique mimimal normal
subgroup $M$. Then $d(G)=\max\{2,d(G/M)\}$.
\label{t:2gen}
\end{thm}

\begin{cor}\label{2hom}
If $G$ is a finite $2$-homogeneous permutation group, then $d(G)=2$.
\end{cor}

\begin{proof}
If $G$ is almost simple, then it satisfies the conditions of the theorem,
with $M$ the simple socle. Since, in the case of socle $\mathrm{PSL}(d,q)$,
the group $G$ contains no graph automorphisms, we have $d(G/M)\le2$ in
all cases, so $d(G)=2$.

If $G$ is affine, its unique minimal normal subgroup $M$ is elementary
abelian, and the quotient $H$ is a linear group; the relevant groups can be
found in \cite{cam,dixon}. If the linear group has
normal subgroup $\mathrm{SL}(d,q)$, $\mathrm{Sp}(d,q)$ ($d>1)$ or $G_2(q)$,
then another application of Theorem~\ref{t:2gen} shows that $d(H)=2$, whence
$d(G)=2$. For $1$-dimensional semi-affine groups, the linear group is metacyclic, and the result is clear. The finitely
many cases remaining can be dealt with case by case: in each case, explicit generators for the linear group are known, and where more than two are given it suffices to show that the corresponding linear group can be generated by two elements. The groups (apart from the sharply $2$-transitive group of degree $59^2$
with linear group $\mathrm{SL}(2,5)\times C_{29}$, which is clearly $2$-generated), are within reach of \textsf{GAP}; the computation can be speeded up by taking the first potential generator to belong to a set of conjugacy class representatives.
\end{proof}

Regarding  the number of orbits on $(n-k)$-partitions, we start by the large values of $k$.

\begin{thm}
Suppose that $G$ is a permutation group of degree $n$ which is 
$k$-homogeneous, where either $6\le k\le n/2$, or $k=5$, $n\ge25$, or $k=4$,
$n\ge34$. Then: 
\begin{enumerate}
\item $G$ has $p(k)$ orbits on the set of $(n-k)$-partitions, where
$p$ is the partition function; 
\item there is one orbit on partitions of each
possible type.
\end{enumerate}
\end{thm}

\begin{proof}
It follows from the Classification of Finite Simple Groups and known results
about $4$- and $5$-homogeneous groups that any $G$ under the assumptions of the theorem is $S_n$ or $A_n$. The second
assertion is a well-known fact about $A_n$ and $S_n$. For the first, given a partition of $\{1,\ldots,n\}$
with $n-k$ parts, for $k\le n/2$, subtracting one from the size of each part
gives a partition of $k$, and every partition of $k$ arises in this way; all set partitions of $\{1,\ldots,n\}$ corresponding
to each fixed partition lie in the same orbit of the symmetric or alternating
group.
\end{proof}

The numbers of orbits of the finitely many $k$-homogeneous groups other than
symmetric or alternating groups for $k=5$ and $k=4$ can be computed.

\begin{thm}
Let $G$ be a $k$-homogeneous group of degree~$n$ (where $k=4$ or $5$ and
$n\ge2k$), other than $S_n$ or $A_n$. Then the number of orbits of $G$ on
$(n-k)$-partitions are given in Tables \ref{t:4hom} and \ref{t:5hom} below.
%(Parts of size $1$ in the partitions are not shown.)
\end{thm}

\paragraph{Remark} We are grateful to Robin Chapman for independent 
confirmation of the values for $\mathrm{P}\Gamma\mathrm{L}(2,32)$ in
Table~\ref{t:4hom}.
{\tiny
\begin{table}[htbp]
\[
\begin{array}{|l||r|r|r|r|r|r|r|}
\hline
\hbox{Degree} & 9 & 9 & 11 & 12 & 23 & 24 & 33 \\
\hline
\hbox{Group} & \mathrm{PSL}(2,8) & \mathrm{P}\Gamma\mathrm{L}(2,8) & M_{11} &
M_{12} & M_{23} & M_{24} & \mathrm{P}\Gamma\mathrm{L}(2,32) \\
\hline
(5,1,\ldots) & 1 & 1 & 2 & 1 & 2 & 1 & 3 \\
(4,2,1,\ldots) & 4 & 2 & 3 & 2 & 4 & 2 & 112 \\
(3,3,1,\ldots) & 4 & 2 & 2 & 2 & 3 & 2 & 82 \\
(3,2,2,1,\ldots) & 12 & 4 & 8 & 3 & 11 & 3 & 2772 \\
(2,2,2,2,1,\ldots) & 5 & 3 & 6 & 5 & 18 & 7 & 9191 \\
\hline
\hbox{Total} & 26 & 12 & 21 & 13 & 38 & 15 & 12160 \\
\hline
\end{array}
\]
\caption{\label{t:4hom}Orbits of $4$-homogeneous groups on $(n-4)$-partitions}
\end{table}
}

\begin{table}[htbp]
\[
\begin{array}{|l||r|r|}
\hline
\hbox{Degree} & 12 & 24 \\
\hline
\hbox{Group} & M_{12} & M_{24} \\
\hline
(6,1,\ldots) & 2 & 2 \\
(5,2,1,\ldots) & 2 & 3 \\
(4,3,1,\ldots) & 2 & 3 \\
(4,2,2,1,\ldots) & 5 & 8 \\
(3,3,2,1,\ldots) & 5 & 8 \\
(3,2,2,2,1,\ldots) & 8 & 22 \\
(2,2,2,2,2,1,\ldots) & 6 & 31 \\
\hline
\hbox{Total} & 30 & 77 \\
\hline
\end{array}
\]
\caption{\label{t:5hom}Orbits of $5$-homogeneous groups on $(n-5)$-partitions}
\end{table}

The situation is very different for the $2$- and $3$-homogeneous groups, to
which we now turn. The main difference is that there are infinitely many
such groups (apart from the symmetric and alternating groups), so there is
no reason why the number of orbits on $(n-k)$-partitions should be bounded
(and indeed it is not; it can grow as a polynomial in $n$, whose degree
depends on $k$ and on the partition considered).

\begin{thm}\label{2.7}
Let $G$ be a $2$-homogeneous permutation group on the set $\{1,\ldots,n\}$. Then
the number of $G$-orbits on the set of partitions of shape 
$(3,1,\ldots,1)$ is $O(n)$, and the number of orbits
on the set of partitions of shape $(2,2,1,\ldots,1)$ is $O(n^2)$.
\end{thm}

\begin{proof}
Since each $2$-set lies in $n-2$ sets of size $3$, $G$ has at most $n-2$ orbits on
$3$-sets. Also, for any $2$-set, there are at most ${n-2\choose2}$
$2$-sets disjoint from it, so there are at most this many orbits on 
$(2,2,1,\ldots,1)$ partitions.
\end{proof}

The bound on the number of orbits is best possible, as the next
example shows.

\begin{example}\label{expl2.1}	
Let $p$ be a prime congruent to $-1$ (mod~$12$). Let $G$
be the group of order $p(p-1)/2$ consisting of all maps of the field of
integers mod $p$ of the form $x\mapsto ax+b$, where $a$ is a non-zero square.
Its normalizer is the group of order $p(p-1)$, consisting of all maps of the
above form for arbitrary non-zero $p$.

The group $G$ is $2$-homogeneous so we take $k=2$. Now the $(p-k)$ partitions
have the form $(3,1,\ldots,1)$ or $(2,2,1,\ldots,1)$. Since $|G|$ is coprime
to $6$,
no element of $G$ except the identity fixes such a partition, and so the
number of orbits is
\[\frac{{p\choose3}+3{p\choose4}}{p(p-1)/2}=\frac{3p^2-11p+10}{12}.\]
Of these, $(p-2)/3$ are on partitions of type $(3,1,\ldots,1)$, and
$(p-2)(p-3)/4$ are on partitions of type $(2,2,1,\ldots,1)$.\qed
\end{example}

\medskip

There is a theorem for 3-homogeneous groups similar to Theorem \ref{2.7}.

\begin{thm}
Let $G$ be a $3$-homogeneous permutation group on the set $\{1,\ldots,n\}$. Then
the number of $G$-orbits on the set of $(n-3)$-partitions is $O(n)$ for
partitions of type $(4,1,\ldots,1)$, $O(n^2)$ for partitions of type
$(3,2,1,\ldots,1)$, and $O(n^3)$ for partitions of type $(2,2,2,1,\ldots,1)$.
\end{thm}

%However, this is not the end of the story. 
In fact we can say more. From CFSG, we know that, if $G$ is
$3$-homogeneous, then one of the following holds:

\begin{itemize}\itemsep0pt
\item $\mathrm{PSL}(2,q)\le G\le\mathrm{P}\Gamma\mathrm{L}(2,q)$, for some
prime power $q$;
\item $G=\mathrm{AGL}(d,2)$ for some $d$;
\item $G$ is one of finitely many exceptions.
\end{itemize}

In the first case, the order of $G$ is $O(n^3)$, so the number of orbits
on partitions of shape $2^31^{n-6}$ will be $\Omega(n^3)$. However, in the
other two cases, the number of orbits is bounded by a constant, independent
of $d$ in the second case. This is clear for the third case, so consider the
second. Suppose we have an $(n-3)$-partition of $\{1,\ldots,n\}$. Then the 
set of points lying in parts of size greater than $1$ has cardinality at 
most $6$, and so these points lie in an affine subspace of dimension at
most $5$. The group is transitive on affine subspaces of any given dimension,
and the stabiliser of such a subspace has only a bounded number of orbits on
its subsets of size at most $6$. 
The number of orbits for this type can be calculated by looking at 
$\mathrm{AGL}(5,2)$.  We find that the
number of orbits on $\mathrm{AGL}(d,2)$ on $(n-3)$-partitions is $12$ for
$d\ge5$. The numbers of orbits on partitions of the different types is given
in Table~\ref{t:3hom}, with the same conventions as earlier.

\begin{table}[htbp]
\[
\begin{array}{|c||c|c|c|}
\hline
\hbox{Degree} & 2^d\ (d\ge5) & 16 & 8 \\
\hline
\hbox{Group} & \mathrm{AGL}(d,2) & \mathrm{AGL}(4,2) & \mathrm{AGL}(3,2) \\
\hline
\relax(4,1,\ldots) & 2 & 2 & 2 \\
\relax(3,2,1,\ldots) & 3 & 3 & 2 \\
\relax(2,2,2,1,\ldots) & 7 & 6 & 3 \\
\hline
\hbox{Total} & 12 & 11 & 7 \\
\hline
\end{array}
\]
\caption{\label{t:3hom}Orbits of $\mathrm{AGL}(d,2)$ on $(n-3)$-partitions}
\end{table}

Similar data can be produced for any finite number of the other $3$-homogeneous
groups. Table \ref{t:3hom2} gives a selection of $3$-homogeneous groups of
degree $n\ge7$, which includes all the sporadic examples, all $4$-homogeneous groups, and all examples with $n\le10$.

\begin{table}[htpb]
\[
\begin{array}{|r|c|r|r|r|r|}
\hline
\hbox{Degree} & \hbox{Group} & (4,1,\ldots) & (3,2,1,\ldots) & (2,2,2,1,\ldots) & 
\hbox{Total} \\
\hline\hline
8 & \mathrm{AGL}(1,8) & 2 & 10 & 11 & 23 \\
  & \mathrm{A}\Gamma\mathrm{L}(1,8) & 2 & 4 & 5 & 11 \\
  & \mathrm{PSL}(2,7) & 3 & 4 & 7 & 14 \\
  & \mathrm{PGL}(2,7) & 2 & 3 & 5 & 10 \\
\hline
9 & \mathrm{PSL}(2,8) & 1 & 4 & 7 & 12 \\
  & \mathrm{P}\Gamma\mathrm{L}(2,8) & 1 & 2 & 3 & 6 \\
\hline
10 & \mathrm{PGL}(2,9) & 2 & 5 & 12 & 19 \\
   & M_{10} & 2 & 5 & 9 & 14 \\
   & \mathrm{P}\Gamma\mathrm{L}(2,9) & 2 & 4 & 8 & 14 \\
\hline
11 & M_{11} & 1 & 2 & 4 & 7 \\
\hline
12 & M_{11} & 2 & 4 & 6 & 12 \\
   & M_{12} & 1 & 1 & 3 & 5 \\
\hline
16 & 2^4:A_7 & 2 & 4 & 10 & 16 \\
\hline 
22 & M_{22} & 2 & 5 & 11 & 18 \\
   & M_{22}:2 & 2 & 4 & 10 & 16 \\
\hline
23 & M_{23} & 1 & 2 & 3 & 6 \\
\hline
24 & M_{24} & 1 & 1 & 2 & 4 \\
\hline
33 & \mathrm{P}\Gamma\mathrm{L}(2,32) & 1 & 16 & 127 & 144 \\
\hline
\end{array}
\]
\caption{\label{t:3hom2}Orbits of $3$-homogeneous groups on $(n-3)$-partitions}
\end{table}

\medskip

\section{Orbits of normalizers}\label{orbs and norms}

In this section we will be interested in the following questions: 
\begin{enumerate}
\item\label{q1} Given an orbit of the
$k$-homogeneous group $G$ on $(n-k)$-partitions,
what is the subgroup of the normalizer of $G$, in $S_n$, which
fixes that orbit? (The question is well-posed since $N_{S_n}(G)/G$ acts on
the set of orbits.) 
\item\label{q2} In particular, for which groups is it the case that every
orbit on $(n-k)$-partitions is invariant under the normalizer,
that is, the action of $N_{S_n}(G)/G$ on the set of orbits is trivial?
\end{enumerate}

If $G$ is the alternating group, then each of its orbits is stabilised by the
symmetric group. For $k\ge4$, any other $k$-homogeneous group is equal to its
normalizer, except for $\mathrm{PGL}(2,8)$ with $n=9$. 
This group is
$5$-homogeneous, and so has the same orbits on partitions of type $(5,1,1,1,1)$ as its
normalizer. Computation shows that this is not the case for other types of
$5$-partitions. So we have the following theorem.

\begin{thm}
Let $k\ge4$ and $n\ge2k$, and let $G$ be a $k$-homogeneous group of degree $n$.
Then $G$ has the same orbits on $(n-k)$-partitions of any given type as its
normalizer, except in the case of $\mathrm{PGL}(2,8)$, for which this 
assertion holds for partitions of type $(5,1,1,1,1)$ but for no other types.
\end{thm}

For $k=3$, we have the following result. We shall say that the pair
$(G,\lambda)$ is \emph{closed} if each orbit of $G$ on $\lambda$-partitions is invariant under $N_{S_n}(G)$. Note that there are three types of partition to be
considered, namely $(4,1,\ldots,1)$, $(3,2,1,\ldots,1)$, and
$(2,2,2,1,\ldots,1)$. Note also that $(G,\lambda)$ is trivially closed if $G=N_{S_n}(G)$.

For $q$ an odd prime power, $q=p^d$, the quotient $\pgaml(2,q)/\psl(2,q)$ is
isomorphic to $C_2\times C_d$, where the factors are generated by a diagonal
automorphism (an element of $\pgl(2,q)$ with non-square determinant) and the
Frobenius field automorphism. If $d$ is even (so that $q$ is a square), this
group contains three subgroups of index~$2$. Two of these are, respectively,
$\psigl(2,q)$, and the group generated by $\pgl(2,q)$ and the square of the
Frobenius automorphism.  We use $\pxl(2,q)$ to denote the third subgroup
of index $2$, obtained by adjoining to $\psl(2,q)$ the
product of diagonal and Frobenius automorphisms. (When $q=9$
this group is better known as $M_{10}$, the point stabiliser in the Mathieu
group $M_{11}$.)

The main theorem of this section is the following.

\begin{thm}\label{main3}
Suppose that $G$ is a $3$-homogeneous subgroup of $S_n$, and $\lambda$ a type 
of $(n-3)$-partitions. Then $(G,\lambda)$ is closed if one of the
following holds:
\begin{itemize}\itemsep0pt
\item $G=N_{S_n}(G)$;
\item $G=A_n$;
\item $\lambda=(4,1,\ldots,1)$ and $G=\agl(1,8)$, $\pgl(2,8)$, $\pgl(2,9)$, $M_{10}$, $\psl(2,11)$,
$M_{22}$, $\pxl(2,25)$, or $\pxl(2,49)$.
\end{itemize}
No other $3$-homogeneous groups appear in a closed pair, with the possible exception of $\pxl(2,q)$ for $q\ge169$.
\label{t:norm}
\end{thm}

{This theorem answers question (\ref{q2}) at the beginning of this section, with the exception of the groups $\pxl(2,q)$ referred to in its statement.} Concerning question (\ref{q1}) the situation may be much more complicated as the following example shows.

\begin{example}
Let $n=17$, and  let $G$ be the $3$-homogeneous group
$\mathrm{PSL}(2,16)$. The normalizer of $G$ in $S_n$ is
$\mathrm{P}\Gamma\mathrm{L}(2,16)=G:4$, with one intermediate subgroup $G:2$.
Table~\ref{t:orbits} gives the number of $G$-orbits on the $14$-partitions of
various types, and the numbers with each of the three possible stabilisers.

\begin{table}[htbp]
\begin{center}
\small
\begin{tabular}{|c||c|c|c|c|}
\hline
Partition & $(4,1,\ldots,1)$ & $(3,2,1,\ldots,1)$ & $(2,2,2,1,\ldots,1)$ &
Total \\
\hline
Orbits & $3$ & $19$ & $72$ & $94$  \\
\hline
Stabiliser $G$ & $0$ & $12$ & $60$ & $72$ \\
\hline
Stabiliser $G:2$ & $2$ & $6$ & $10$ & $18$ \\
\hline
Stabiliser $G:4$ & $1$ & $1$ & $2$ & $4$ \\
\hline
\end{tabular}
\end{center}
\caption{\label{t:orbits}Stabilisers of orbits of $\mathrm{PSL}(2,16)$}
\end{table}
For the group $G=\mathrm{PSL}(2,2^p)$, with $p$ prime and $p>3$, the situation
is much simpler: no $(2^p-2)$ partition can be fixed by an element outside
$G$, and so every orbit has stabiliser $G$. (This also shows, for example, that
the numbers of orbits for $\mathrm{PSL}(2,32)$ are five times those for
$\mathrm{P}\Gamma\mathrm{L}(2,32)$ given in Table~\ref{t:3hom2}.)\qed
\end{example}

We now give the proof of Theorem~\ref{t:norm}.

\begin{pf}
We begin by listing the $3$-homogeneous groups.
\begin{itemize}\itemsep0pt
\item[(a)] $S_n$, $A_n$.
\item[(b)] (Some) subgroups of $\pgaml(2,q)$ containing $\psl(2,q)$, for $q$
a prime power (``some'' means ``all'' if and only if $q$ is even or congruent
to $3$ mod~$4$).
\item[(c)] $\agl(d,2)$.
\item[(d)] Finitely many ``sporadic'' examples: $\agl(1,8)$, $\agaml(1,8)$,
$M_{11}$, $M_{11}$ (degree~$12$), $M_{12}$, $2^4:A_7$, $M_{22}$, $M_{22}:2$,
$M_{23}$, $M_{24}$, and $\agaml(1,32)$.
\end{itemize}

We remark that, of these, the groups which are equal to their normalizers
(and so fall in the first case in the Theorem) are:
\begin{itemize}\itemsep0pt
\item[(a*)] $S_n$.
\item[(b*)] $\pgaml(2,q)$.
\item[(c*)] $\agl(d,2)$.
\item[(d*)] All except $\agl(1,8)$ and $M_{22}$.
\end{itemize}

Now type (a) are always closed. Type (c), and also type (d)
with the exception of $\agl(1,8)$ and $M_{22}$, are equal to their
normalizers, so are trivially closed. For the remaining cases in (d),
computation shows that, for $G=\agl(1,8)$ (degree $8$) or $G=M_{22}$ (degree
$22$), and $\lambda$ is a partition type of rank $n-3$, then $(G,\lambda)$ is
closed if and only if $\lambda=(4,1,\ldots,1)$.

In fact, the numbers of orbits on the three types of partitions for $G$ and
its normalizer are $(2,10,11)$ and $(2,4,5)$ for $G=\agl(1,8)$, and
$(2,5,11)$ and $(2,4,10)$ for $G=M_{22}$. Note that $M_{22}$ comes very
close: only two orbits of each of the other two types are fused by 
$M_{22}:2$.

So it remains to deal with type (b).

\paragraph{Subgroups containing $\pgl(2,q)$} $\phantom{.}$

\subparagraph{Partitions of type $\lambda=(4,1,\ldots,1)$} $\phantom{.}$

These partitions correspond naturally to $4$-subsets. Now orbits
of $\pgl(2,q)$ on $4$-tuples are parametrised by \emph{cross ratio}: there is
some flexibility about the definition, but I will assume that the cross ratio
of $(\infty,0,1,a)$ is $a$. Now the $24$ orderings of a $4$-set give rise to
a set of $6$ cross ratios (or occasionally fewer) of the form
\[\{z,1/z,1-z,1/(1-z),z/(z-1),(z-1)/z\}.\] So $\GF(q)\setminus\{0,1\}$ is
partitioned into sets of $6$ (or fewer) corresponding to orbits of
$\pgl(2,q)$ on $4$-sets.

Now $\pgaml(2,q)$ is generated by $\pgl(2,q)$ and the \emph{Frobenius map}
$x\mapsto x^p$, where $q$ is a power of $p$, say $q=p^t$. Thus there is a
cyclic group of order $t$ permuting the orbits (or the sets as above). To
show that no proper subgroup of $\pgaml(2,q)$ containing $\pgl(2,q)$ is
good, it suffices to find a $6$-set which is fixed by no power of the
Frobenius map except the identity.

Suppose that we have a $6$-set $\{z,1/z,1-z,1/(1-z),z/(z-1),(z-1)/z\}$ which
is fixed by a non-trivial power of the Frobenius map; we can assume that this
has the form $x\mapsto x^{p^u}$ where $u$ divides $t$. We put $t=uv$ and
$p^u=r$, so that $q=r^v$ and the map under consideration has fixed field
$\GF(r)$. Now, for every $z$, $z^r\in\{z,1/z,1-z,1/(1-z),z/(z-1),(z-1)/z\}$.
There are six possibilities:
\begin{itemize}\itemsep0pt
\item $z^r=z$. Then $z\in\GF(r)$.
\item $z^r\in\{1/z,1-z,z/(z-1)\}$. In each of these cases, we find that 
$z^{r^2}=z$, so $z\in\GF(r^2)$. So we may assume that $q=r^2$.
\item $z^r\in\{1/(1-z),(z-1)/z\}$. In these cases, we find that 
$z^{r^3}=z$; so we may assume that $q=r^3$.
\end{itemize}
Only one of these possibilities can hold. We may assume that $q\ne r$, so that
$z$ can be chosen so that the first possibility does not hold.

Suppose that $q=r^2$. Now the above argument shows that the $r^2-r$ elements
outside $\GF(r)$ satisfy one of the three equations $z^r=1/z$, $z^r=1-z$, or
$z^r=z/(z-1)$. These are polynomials of degrees $r+1$, $r$, $r+1$ respectively;
so $3r+2\ge r^2-r$, giving $r\le 4$. Now $\pgl(2,4)\cong A_5$ falls under
case (a); and computation shows that $(\pgl(2,9),\lambda)$ is closed but
$\pgl(2,16)$ and $\pgl(2,16):2$ are not. (Both the last two groups have three
orbits on $4$-sets, but $\pgl(2,16):4$ has only two.)

Now suppose that $q=r^3$. We argue similarly to say that the elements outside
$\GF(r)$ satisfy one of the two equations $z^r=1/(1-z)$ or $z^r=(z-1)/z$,
both polynomials of degree $r+1$. Thus $r^3-r\le2(r+1)$, with only the
solution $r=2$. The pair $(\pgl(2,8),\lambda)$ is good, since $\pgl(2,8)$ is
$4$-homogeneous.

For the other two partition types, the argument is less elegant.

\subparagraph{Partitions of type $\lambda=(3,2,1,\ldots,1)$} $\phantom{.}$

In this case, each orbit has a
representative in which the $3$-set is $\{\infty,0,1\}$, by $3$-transitivity
of $G$. The elements of $\pgl(2,q)$ which map this set to itself are 
the maps $z\mapsto f(z)$, where $f(z)$ is one of the six linear fractional
expressions which came up in our discussion of cross ratio. Moreover, all
three points are fixed by the Frobenius map.

Suppose that $p$ is odd. Take $x\in\GF(p)\setminus\{0,1\}$ and $y$ in no
proper subfield of $\GF(q)$, and consider the partition as above whose $2$-set
is $\{x,y\}$.

The points fixed by the above transformations are $x=-1$, $x=2$,
$x=\frac{1}{2}$,
and $x$ a primitive $6$th root of $1$. If $p\ne3$ or $7$ we can choose $x$ to
satisfy none of these, so there is only one set in the orbit. But if we
choose $y$ to be a primitive element, then it is not fixed by any power of the
Frobenius map, so this set is in a regular orbit of this map.

If $p=3$, then $x=2$ is fixed by three maps $z\mapsto 1/z$, $z\mapsto1-z$,
and $z\mapsto z/(z-1)$. So if the orbit is fixed by the Frobenius map, then
$y$ must satisfy $y^r\in\{1/y,1-y,y/(y-1)\}$.
There are at most $3r+3$ such elements. So $r^v-r\le 3r-3$, whence $r=2$.
But the computer establishes that not all $\pgl(2,9)$-orbits are fixed
by $\pgaml(2,9)$.

Suppose that $p=7$. A similar but easier argument applies, since
each of $2$, $4$ and $6$ is fixed by just a single element, so we find
$r^v-r\le r+1$, which is impossible.

Lastly we have the case $p=2$. We may assume that $t>2$, since $\pgl(2,4)
\cong A_5$. Now if $y\ne 1/x,1-x,x/(x-1)$, then only the identity in $\pgl(2,q)$
fixes this partition. Choosing $y$ to be a primitive element of $\GF(q)$ shows
that the group generated by the Frobenius map acts regularly on the orbit of
this partition.

\subparagraph{Partitions of type $(2,2,2,1,\ldots,1)$} $\phantom{.}$

We can assume that an orbit we are considering contains the partition
$\{\infty,0\}$, $\{1,a\}$ and $\{b,c\}$.

Suppose first that $p>2$, and take $a=2$. The three linear fractional
transformations fixing the pair of sets making up the first two cycles are
$z\mapsto2/z$, $z\mapsto(z-2)/(z-1)$, and $z\mapsto2(z-1)/(z-2)$, have among
them at most six fixed points, namely $\pm\sqrt{2}$, $1\pm\sqrt{-1}$, and
$2\pm\sqrt{-2}$; so there is a point $a$ fixed by none of these. If we
take $b$ and $c$ to be linearly independent over $\GF(p)$, then only the
identity in $\pgl(2,q)$ fixes the three sets; and if we take $z\ne y^p$,
then we find that they are not fixed by any power of the Frobenius map.

If $p=2$, the argument is similar. If $t$ is even, then we have a subfield
$\GF(4)$; if $t$ is divisible by $3$, a subfield $\GF(8)$. If neither of
these occurs, then no field automorphism can fix a partition of shape $\lambda$,
since only permutations of order dividing $48$ can do so.

\paragraph{Groups not containing $\pgl(2,q)$} $\phantom{.}$

If $q$ is not a square, then any subgroup of $\pgaml(2,q)$ containing
$\psl(2,q)$ but not $\pgl(2,q)$ must lie inside $\psigl(2,q)$. Moreover, we
may assume that $G=\psigl(2,q)$, since any other subgroup is contained in
a group twice as large which does itself contain $\pgl(2,q)$. This case only
arises if $q\equiv3$ (mod~$4$), since otherwise $\psigl(2,q)$ is not
$3$-homogeneous.

If the $\psigl(2,q)$-orbit of a partition $P$ is fixed by $\pgaml(2,q)$,
then a set in that orbit must be fixed by an element of
$\pgaml(2,q)\setminus\psigl(2,q)$, since then its stabiliser will be twice as
large, and the orbit the same size. We can assume that such an element has
$2$-power order, and all its cycles have the same size (since $\psigl(2,q)$
has odd order). This excludes shape $(3,2,1,\ldots,1)$, so we have to consider
the other two types. Moreover, $q$ is an odd power of $p$, so these maps 
do not involve field automorphisms.

First consider type $(4,1,\ldots,1)$, so we are looking for an element fixing
a $4$-set, acting on it as either a double transposition or a $4$-cycle.
By $3$-homogeneity, we can consider $4$-sets of the form $\{\infty,0,1,a\}$.
For a double transposition. There are three possibilities:
\begin{itemize}\itemsep0pt
\item $(\infty,0)(1,a)$: $z\mapsto a/z$ does this. Its determinant is $-a$,
which is a nonsquare if and only if $a$ is a square.
\item $(\infty,1)(0,a)$: $z\mapsto(z-a)/(z-1)$ does this. Its determinant
is $-1+a$, which is a nonsquare if and only if $1-a$ is a square.
\item $(\infty,a)(0,1)$: $z\mapsto(az-a)/(z-a)$ does this. Its determinant
is $-a^2+a$, which is a nonsquare if and only if $a(a-1)$ is a square.
\end{itemize}
Now the product of these three numbers is $-a^2(a-1)^2$, which is a nonsquare.
So $0$ or $2$ of them are squares for every $a$. Indeed, it is well-known
(from the construction of the Paley design) that there are $(q+1)/4$ elements
$a$ for which $a$ and $1-a$ are both non-squares.
Now we consider $4$-cycles. Up to inversion, there are three possibilities:
\begin{itemize}\itemsep0pt
\item $(\infty,0,1,a)$: $z\mapsto 1/(cz+1)$, where $1/(c+1)=a$ and $ac+1=0$;
these equations have a unique solution $a=2$.
\item $(\infty,0,a,1)$: $z\mapsto a/(cz+1)$, where $a/(ac+1)=1$ and $c+1=0$;
the solution is $a=\frac{1}{2}$.
\item $(\infty,1,0,a)$: $z\mapsto (z-1)/(z+c)$, where $-1/c=a$ and $a+c=0$;
the solution is $a=-1$.
\end{itemize}
So, if every orbit is accounted for, we have $(q+1)/4\le 3$, so $q=7$ or
$q=11$. It can be checked (by hand or by computer) that $\psl(2,7)$ for type $(4,1,\dots,1)$ is not closed, but $\psl(2,11)$ is.

Now consider the type $(2,2,2,1,\ldots,1)$. This time we can assume that the
partitioned $6$-set is $\{\{\infty,0\},\{1,a\},\{b,c\}\}$, and it is fixed
by an involution, which fixes one or all of the $2$-sets. If there is a
cycle $(\infty,0)$, then $1$ maps to $a$, $b$ or $c$, and we find that the
product of any two of $a,b,c$ is the third. (For example, if $(1,a)$ is a
cycle, then the map is $z\mapsto a/z$, and $a/b=c$.)

In the remaining case, we have four triple transpositions to consider, namely
$(\infty,1)(0,a)(b,c)$, $(\infty,a)(0,1)(b,c)$, $(\infty,b)(0,c)(1,a)$, or
$(\infty,c)(0,b)(1,a)$. The third and fourth are equivalent under interchange
of $b$ and $c$. We have:
\begin{itemize}\itemsep0pt
\item $(\infty,1)(0,a)(b,c)$: we find $1-a=(1-b)(1-c)$.
\item $(\infty,a)(0,1)(b,c)$: we find $(a-b)(a-c)=a(1+a)$.
\item $(\infty,b)(0,c)(1,a)$: we find $a(b-c)=a-b$.
\end{itemize}
In each case, given $a$ and $b$, there is only one choice of $c$, so $q\le 5$,
a contradiction.

\medskip

The remaining class to be considered are the groups $\pxl(2,q)$. As mentioned
earlier, we have checked by computer the odd prime power squares up to $121$,
and found that $\pxl(2,q)$ acting on $(4,1,\ldots,1)$ partitions is closed
for $q=9$, $25$ and $49$, but not for $q=81$ or $121$.\qed
\end{pf}

\section{Groups having only one orbit in a given kernel type}\label{sem1}

 The remainder of this paper is dedicated to the application to semigroup theory of the results found above.
We want to describe the structure (elements, ranks, automorphisms, congruences, regularity, idempotent generation, etc.) of semigroups generated by a $k$-homogenous subgroup of $S_n$  and some singular maps of rank larger than $n/2$. We will use several times the well known fact (\cite[p.11]{Ho95}) that if $S$ is a finite semigroup and $a\in S$, then there exists a natural number $\omega$ such that $a^\omega$ is idempotent. 

In this section we are going to study the semigroups generated by a singular transformation $t$, such that $\rank(t)\ge n/2$, and a permutation group that has only one orbit on the kernel type of $t$.  
 
We start by noting the following. Suppose that the kernel of $t$ has type $(l_1,\ldots,l_k)$ and $m$ is the largest natural such that $l_m>1$. Then $G$ must be $(\sum^{m}_{i=1} l_i)$-homogeneous and hence, given that the rank of $t$ is $k$, the group must be $p$-homogeneous, for some $p\in \{n-k+1,\ldots,2(n-k)\}$. The smallest value of $p$ is attained if the kernel type is $(n-k+1,1,\ldots,1)$, and the largest value is attained for kernel type $(2,\ldots,2,1,\ldots,1)$. Therefore, for any practical considerations we might assume that our groups are $(n-k+1)$-homogeneous. 

\begin{thm}\label{main3.3}
Let $t$ be a singular map in $T(X)$, with $X=\{1,\ldots,n\}$, and suppose that $t$ has kernel type $(l_1,\ldots,l_k)$, with $k\ge n/2$; let $G$ be a group having only one orbit in the partitions of that type. { Let $E$ denote the set of idempotents of $\langle t,G\rangle\setminus G$.} Then $$\langle t,G\rangle\setminus G =\langle t,S_n\rangle\setminus S_n =\langle E,t\rangle.$$
\end{thm}

The proof of this theorem will follow from a sequence of lemmas. Throughout this section $t\in T_n$ will be a rank $k$ map of kernel type $(l_1,\ldots,l_k)$, and $G\le S_n$ will be a $(n-k+1)$-homogeneous group  having only one orbit on the partitions of type  $(l_1,\ldots,l_k)$. 

We now introduce some notation. Given the rank and the kernel type of $t$ we have
\[
{t}=\left(
\begin{array}{ccc}
A_{1}&\cdots&A_{k}\\
a_{1}&\cdots&a_{k}
\end{array}\right), 
\] where $|A_i|=l_i$ (for all $i\in \{1,\ldots,k\}$). 

Throughout this section we will assume that the fixed map $t$ has  kernel $T=(A_1,\ldots A_k)$ of type $(l_1,\ldots ,l_k)$.

Observe that for every $g,h\in G$ we have 
\[
{g^{-1}th}=\left(
\begin{array}{ccc}
A_{1}g&\cdots&A_{k}g\\
a_{1}h&\cdots&a_{k}h
\end{array}\right). 
\]
Since $k\ge n/2$ and the group is $(n-k+1)$-homogeneous, it follows that the group is also $k$-homogeneous. Thus given any $k$-set $Y$ contained in $X$, there exists  $th\in \langle G,t\rangle$ such that $Xth=Y$. Similarly, given any partition $Q=(B_1,\ldots,B_k)$ of $X$ of type $(l_1,\ldots,l_k)$, since $G$ has only one orbit on the partitions of this type, it follows that there exists $g\in G$ such that $\{A_1,\ldots,A_k\}g=\{B_1,\ldots,B_k\}$ and hence the kernel of $g^{-1}t$ is $(B_1,\ldots,B_k)$. This proves the following lemma. 

\begin{lem}\label{lem1a)}
Given any partition $Q$ of type $(l_1,\ldots,l_k)$ and any $k$-set $Y\subseteq X$, there exist $g,h\in G$ such that $\ker(g^{-1}th)=Q$ and $Xg^{-1}th=Y$.
\end{lem}

The previous result shows that $\langle G,t\rangle$ has rank $k$ maps of every possible image and kernel. The next result provides the anlogous result for idempotents.

\begin{lem}\label{lem1}
Given any partition $Q$ of type $(l_1,\ldots,l_k)$ and any $k$-set $Y\subseteq X$ such that $Y$ is a transversal for $Q$, there exists an idempotent $e\in \langle t,G\rangle$ such that $\ker(e)=Q$ and $Xe=Y$.
\end{lem}
\begin{pf}
{ By the previous lemma we know that there exist $g,h\in G$ such that $\ker(g^{-1}th)=Q$ and $Xg^{-1}th=Y$. Since $Y$ is a transversal for $Q$, it follows that there exists $k\in G$, namely $k:=hg^{-1}$, such that $\rank(tkt)=\rank(t)$. Therefore, every element in $\langle tk\rangle :=\{(tk)^i\mid i\in \mathbb{N}\}$, the monogenic semigroup generated by $tk$, has the same rank as $t$. Since every finite semigroup contains an idempotent, we conclude that $\langle tk\rangle$ contains an idempotent, say $(tk)^\omega$, and hence
%all elements $(g^{-1}th)^i$ (for $i\in \mathbb{N}$) have the same rank, and hence all of them have the same kernel and the same image. 
%Using Schur's Lemma applied to the set $\{(g^{-1}th)^i\mid i\in \mathbb{N}\}$ exactly as we did in 
  $g^{-1}(tk)^\omega g=g^{-1}(thg^{-1})^\omega g$ is also an idempotent with kernel $Q$ and image $Y$. The result follows.\qed 
}\end{pf}

In order to increase the readability of the arguments we introduce some notation. 
Given a partition $P=(A_{1},\ldots,A_{k})$ of $X$,  and a transversal $S=\{a_{1},\ldots,a_{k}\}$ for $P$, where $a_{i}\in A_{i}$ (for all $i$),  we represent $A_i$ by $[a_i]_{_P}$ and
the pair $(P,S)$ induces an idempotent mapping defined by $[a_i]_{_P}e=\{a_i\}$. Conversely, every idempotent can
be so constructed from a partition and a transversal. With this notation, we can write the  idempotent   $$e=\left(\begin{array}{cccccc}
[a_1]_{_P}&\ldots &[a_k]_{_P}\\
a_1      &\ldots & a_k
\end{array}\right)$$ in the more compact form $e=([a_1]_{_P}, \ldots,[a_k]_{_P})$.
This notation extends to $e=([\underline{a_1},b]_{_P},[a_2]_{_P}, \ldots,[a_k]_{_P})$ when  $b\in [a_1]_{_P}$ and
$[a_i]_{_P}e=\{a_i\}$. By $([a_1],\ldots ,[\underline{a_i},{b}], \ldots,[a_k])$ we denote the set of all
idempotents $e\in T_n$ with image $\{a_1,\ldots , a_k\}$ and such that the $\Ker(e)$-class of $a_i$ contains (at least)
two elements: $a_i$ and $b$, where the underlined element (in this case $a_i$) is the image of the class under
$e$. 

\begin{lem}
Let $q_1,q_2\in \langle G,t\rangle$ be two maps of rank $k$ such that  $Xq_1=\{b_1,\ldots, b_k\}$ and $Xq_2=\{b_2,\ldots, b_{k+1}\}$. Then there exists an idempotent $e\in  \langle G,t\rangle$ such that $Xq_1e=Xq_2$.
\end{lem}
\begin{pf}
By the previous lemma,  given any partition of the same type as the kernel of $t$, and any transversal for it, there exists in $\langle t,G\rangle$ an idempotent with that partition as kernel and that transversal as image. Therefore we can  pick a partition of the same type as the kernel of $t$ with the following parts:   $Q_0=\{\{b_1,{b_{k+1}},\ldots\},\{b_2,\ldots\},\ldots ,\{b_{k},\ldots\}\}$; it is clear that  $Xq_2$ is a transversal for $Q_0$ and hence the idempotent $e=[[ b_1,\underline{b_{k+1}}]_{Q_0},[b_2]_{Q_0},\ldots,[b_k]_{Q_0}]$ satisfies the desired $Xq_1e=Xq_2$. \qed
\end{pf}

Since given any two $k$-sets $Y,Z\subseteq X$, there exists a sequence of $k$-subsets of $X$, say $(Y_1,\ldots,Y_m)$, such that $|Y_i\cap Y_{i+1}|=k-1$, with $Y_1=Y$ and $Y_m=Z$, the following result is a consequence of the application of the previous lemma as many times as needed. 

\begin{cor}\label{corol3.7}
Let $q_1,q_2\in \langle G,t\rangle$ be two rank $k$ maps. Then there exists a sequence of idempotents $e_1,\ldots, e_j$ such that $Xq_1e_1\ldots e_j=Xq_2$. 
\end{cor}
 
So far we showed that it is possible to use idempotents to build maps with any given kernel of the same type as the kernel of $t$,  and as  image any $k$-set. Now we move a step forward.

\begin{lem}
Let $p$ be a map of rank $k$ and  $x,y\in Xp=\{p_1,\ldots,p_k\}$. Denote by $(xy)$ the transposition induced by $x$ and $y$. Then there exist idempotents $e_1,e_2,e_3\in \langle G,t\rangle$ such that $p(xy)=pe_1e_2e_3$. 
\end{lem}
\begin{pf} Since $p$ is non-invertible, there exists $c\in X\setminus Xp$.
By Lemma \ref{lem1},  $\langle G,t\rangle$ intersects the following sets of idempotents:  
\[
\begin{array}{l}
A=([p_1],\ldots,[x,\underline{c}],\ldots,[y],\ldots,[p_k]) \\
B=([p_1],\ldots,[y,\underline{x}],\ldots,[c],\ldots,[p_k]) \\
C=([p_1],\ldots,[c,\underline{y}],\ldots,[x],\ldots,[p_k]). 
\end{array}
\]
Taking $e_1\in A$, $e_2\in B$, and $e_3\in C$, all from $\langle G,t\rangle$, we get the desired composition $p(xy)=pe_1e_2e_3$.  \qed
\end{pf}

Now we can prove Theorem \ref{main3.3}.
\begin{pf}
To prove the theorem, observe that $\langle t,S_n\rangle\setminus S_n$ is generated by the set $\{gth\mid g,h\in S_n\}$. If $gth\in \langle t,G\rangle$, for all $g,h\in S_n$, the result would follow. 

Let $Q$ be the partition induced by the kernel of $gth$ and let $S_1$ be a transversal for $Q$. Since $Q$ has the same kernel type as $\ker(t)$ it follows, by Lemma \ref{lem1}, that there exists an idempotent $e\in \langle G,t\rangle$ such that $Xe=S_1$ and the kernel of $e$ is $Q$. Let $S$ be a transversal for the kernel of $t$. By Corollary \ref{corol3.7}, there exists a sequence of idempotents such that $Xee_1\ldots e_i=S$. Thus the map $ee_1\ldots e_it$ has the same rank as  $t$, and the same kernel as $gph$. Similarly, there are idempotents $f_1,\ldots,f_l$ such that $Xee_1\ldots e_itf_1,\ldots,f_l=Xgth$. Thus, there exists a permutation $\sigma$ of the set $Xgth$ such that $ee_1\ldots e_itf_1,\ldots,f_l\sigma=gth$. Therefore, 
\[
Xgth=ee_1\ldots e_itf_1,\ldots,f_l\sigma=ee_1\ldots e_itf_1,\ldots,f_l(x_1y_1)\ldots (x_m y_m),
\]
and each of these transpositions can be replaced by a product of three idempotents of $\langle t,G\rangle$. { This also proves that  $\langle t,G\rangle \setminus G\subseteq \langle E,t\rangle$; as the converse inclusion is obvious, the 
result follows.} \qed
\end{pf}

We recall here some known facts about the semigroups $\langle t,S_n\rangle$. 

\begin{thm}\label{illustr}
Let $a\in T_n$ be singular and let $S=\langle a,S_n\rangle\setminus S_n$. Let $\Omega:=\{1,\ldots,n\}$.  Then 
\begin{enumerate}
\item\label{i1} $S= \{b\in \trans \mid (\exists g\in S_n)\ \ker(a)g\subseteq \ker(b)\}$;
\item\label{i2} $S$ is regular, that is, for all $a\in S$ there exists $b\in S$ such that $a=aba$;
\item\label{i3} $S$ is generated by its idempotents;
\item\label{i4} $S$ and $\langle g^{-1}ag \mid g\in S_n\rangle$ have the same idempotents;
\item\label{i5} $S=\langle g^{-1}ag \mid g\in S_n\rangle$;
\item\label{i7} the automorphisms of $\langle a,S_n\rangle$ are those induced under conjugation by the elements of the normalizer of $S$ in $\sym$, $$\mbox{Aut}(\langle a,S_n\rangle)\cong\mbox{N}_{\sym}(\langle a,S_n\rangle);$$ 
\item\label{i8}  we also have $\mbox{Aut}(\langle a,S_n\rangle)\cong S_n$;
\item\label{i9} all the congruences of $\langle a,S_n\rangle$ are described;
\item\label{i10} if $e^{2}=e\in \langle a,S_n\rangle$, $r:=\rank(e)$, then $$\{f\in \langle a,S_n\rangle \mid \ker(f)=\ker (e) \mbox{ and }\Omega f=\Omega e\}\cong {S}_{r}.$$
\item\label{i11} regarding principal ideals and Green's relations, for all $a,b\in S$, we have \[
\begin{array}{rcl}
aS=bS &\Leftrightarrow & \ker(a)=\ker(b)\\	
Sa=Sb &\Leftrightarrow & \Omega a=\Omega b\\	
SaS=SbS &\Leftrightarrow & \rank (a)=\rank( b )\\
\end{array}
\]
\item\label{i12} { the minimum size of a generating set for $\langle a,G\rangle$, for $a\in T_n\setminus S_n$, is $3$.} 
\item\label{i13} { the minimum size of a set $A$ of rank $k$ maps such that $\langle A,S_n\rangle$ generates all maps of rank at most $k$ is $p(k)$.} 
\end{enumerate}
\end{thm}
\begin{pf}
 Equality (\ref{i1}) was proved by Symons in \cite{symo}. Claims (\ref{i2}), (\ref{i3}) and (\ref{i5}) were proved by Levi and McFadden in \cite{lm}. Claim (\ref{i4}) was proved by McAlister in \cite{mcalister}, and (together with (\ref{i3})) it also implies (\ref{i5}).

 Claim (\ref{i7}) follows from the general result that every automorphism of a  semigroup $S\le T_n$ containing all the constants is induced under conjugation by the normalizer of $S$ in $S_n$  (see \cite{Sullivan} and also \cite{arko,arko2}); since, by (\ref{i1}), the semigroups $\langle S_n,a\rangle$ contain all the constants, the result follows. 
 Claim (\ref{i8}) was proved by Symons in \cite{symo}, but is also an easy consequence from (\ref{i7}).
In \cite{levi00} Levi  described all the congruences of an $S_n$-normal semigroup  and hence described the congruences in $S$. Thus (\ref{i9}).

Statement  (\ref{i10}) belongs to the folklore (see Theorem 5.1.4 of \cite{mar}).  
The results about principal ideals  (\ref{i11}) were proved by Levi and McFadden in \cite{lm}.  

{
Claim (\ref{i12}) follows from the fact that $S_n$ is generated by two elements. 

Regarding (\ref{i13}), observe that given any rank $k$ map $t$ we have that all rank $k$ maps in $\langle t,G\rangle$ have the same kernel type as $t$; conversely, every rank $k$ map of the same kernel type of $p$ belongs to $\langle t,G\rangle$. Therefore, to generate all maps of rank $k$ a necessary and sufficient condition is that there is in $A$  one map of each kernel type, so that $|A|= p(n)$. It is well known that the maps of rank $k$, for $k>1$, generate all maps of smaller ranks. The result follows. }
\qed
\end{pf}

The previous results immediately imply the following.

\begin{thm}
Let $t$ be a singular map in $T_n$, the full transformation monoid on $\Omega:=\{1,\ldots ,n\}$,  and suppose that $t$ has kernel type $(l_1,\ldots,l_k)$, with $k\ge n/2$; let $G$ be a group having only one orbit in the partitions of that type. Let $S=\langle t,G\rangle\setminus G$. Then
\begin{enumerate}
\item\label{i1b} $S= \{b\in \trans \mid (\exists g\in S_n)\ \ker(a)g\subseteq \ker(b)\}$;
\item\label{i2b} $S$ is regular, that is, for all $a\in S$ there exists $b\in S$ such that $a=aba$;
\item\label{i3b} $S$ is generated by its idempotents;
\item\label{i4b} $S$ and $\langle g^{-1}ag \mid g\in G\rangle$ have the same idempotents;
\item\label{i5b} $S=\langle g^{-1}ag \mid g\in G\rangle$;
\item\label{i7b} the automorphisms of $\langle a,G\rangle$ are those induced under conjugation by the elements of the normalizer of $G$ in $S_n$, $$\mbox{Aut}(\langle a,G\rangle)\cong\mbox{N}_{S_n}(G);$$ 
%\item\label{i8b}  we also have $\mbox{Aut}(\langle a,G\rangle)\cong S_n$;
\item\label{i9b} all the congruences of $S$ are described;
\item\label{i10b} if $e^{2}=e\in \langle a,G\rangle$, $r:=\rank(e)$, then $$\{f\in \langle a,G\rangle \mid \ker(f)=\ker (e) \mbox{ and }\Omega f=\Omega e\}\cong \mathcal{S}_{r}.$$
\item\label{i11b} regarding principal ideals and Green's relations, for all $a,b\in S$, we have \[
\begin{array}{rcl}
aS=bS &\Leftrightarrow & \ker(a)=\ker(b)\\	
Sa=Sb &\Leftrightarrow & \Omega a=\Omega b\\	
SaS=SbS &\Leftrightarrow & \rank (a)=\rank( b )\\
\end{array}
\]
\item\label{i12b} the minimum size of a generating set for $\langle a,G\rangle$ is $3$.
\item\label{i13b}   {let $A$ be a set of rank $k$ maps such that $\langle A,G\rangle$ generates all maps of rank at most $k$ and $A$ has minimum size among the sets with that property. } 
{A bound for the size of the sets $A\subseteq T_n$ such that $\langle G,A\rangle$ generate all maps of rank $k$ is given in Table~\ref{t:genset2}. In the
middle column is the type on which $G$ has only one orbit.}
\end{enumerate}
\end{thm}
%\marginpar{Why do we have this Table 6? No idea: it is not referred to. Take it out!}
%\begin{table}[htbp]
%\[
%\begin{tabular}{|c|c|}		
%\hline
%Kernel type of $a$&homog. of $G$\\
%\hline \hline 
%$(2,1,\ldots,1)$&$2$\\ 
%$(2,2,\ldots,1)$&$4$\\ 
%$(3,1,\ldots,1)$&$3$\\ 
%$(3,2,\ldots,1)$&$5$\\ 
%$(4,1,\ldots,1)$&$4$\\
%others     &$n-2$\\ \hline 
%\end{tabular}
%\]
%\caption{\label{t:genset}Numbers of generators}
%\end{table}
{
\begin{table}[htbp]
\[
\begin{array}{|c|c|c|}		
\hline
\hbox{Rank} & \hbox{Kernel type} & |A|\\
\hline \hline 
n-1 & (2,1,\ldots,1) & 1 \\ 
n-2 & (2,2,1,\ldots,1) & 2 \\ 
    & (3,1,\ldots,1) & O(n) \\
\hline 
n-3 & (4,1,\ldots,1) & 144 \\
    & (3,2,1,\ldots,1) & 5  \\
    & (2,2,2,1,\ldots,1) & 3 \\ 
\hline
n-4 & (5,1,\ldots,1) & 15 \\
    & \hbox{other} & 5 \\
\hline    
k\ (n/2\le k\le n-5) & \hbox{any} & p(k) \\
\hline 
\end{array}
\]
\caption{\label{t:genset2}Generating all maps of rank $k$}
\end{table}
}
\begin{pf}
Claims (\ref{i1b})--(\ref{i5b}) and (\ref{i9b})--(\ref{i11b}) all follow from the previous theorem and Theorem \ref{main3.3}.

Regarding claim (\ref{i7b}), observe that by (\ref{i1b}) the semigroup $\langle a,G\rangle$ contains all the constant maps and hence, by \cite[Theorem 1]{Sullivan}, its automorphisms are 
\[
\{\tau^g \mid g\in S_n \wedge g^{-1}\langle a,G\rangle g =\langle a,G\rangle\}, 
\]where, for a given $g\in S_n$, we have $\tau^g:\langle a,G\rangle\rightarrow \langle a,G\rangle$ defined by $f\tau^g=g^{-1}fg$. Note that a permutation $g\in N_{S_n}(G)$ normalizes $\langle a,G\rangle$ if and only if  it normalizes $G$ and  $\langle a,G\rangle\setminus G$. Thus  the automorphisms of $\langle a,G\rangle$ are the maps induced under conjugation by the elements in the normalizer $N_{S_n}(G)$ that also normalize $\langle a,G\rangle\setminus G$.  Since  $\langle a,G\rangle\setminus G=\langle a,S_n\rangle\setminus S_n$ it follows that every permutation of $S_n$ normalizes $\langle a,G\rangle\setminus G$. We conclude that the automorphisms of $\langle a,G\rangle$ are all the maps 
\[
\{\tau^g:\langle a,G\rangle \rightarrow \langle a,G\rangle \mid g\in N_{S_n}(G)\}.
\]

To prove that in fact we have $\Aut(\langle a,G\rangle\setminus G)\cong N_{S_n}(G)$ we only need to observe that primitive groups have trivial center (that is, only the identity  in $G$ commutes with all other elements of $G$).

Regarding (\ref{i12b}), observe that every two homogeneous group is $2$-generated (Corollary \ref{2hom}).

Finally, (\ref{i13b}), follows from the results in the previous section, with a little care. For example, a permutation group transitive on partitions of type 
$(2,2,1,\ldots,1)$ is $4$-homogeneous, and so $3$-homogeneous; so
it is transitive on $(3,1,\ldots,1)$ partitions also. A group transitive on
$(3,2,1,\ldots,1)$ partitions is $5$-homogeneous, and so symmetric, alternating
or a Mathieu group; we refer to Table~\ref{t:3hom2} for the Mathieu groups 
(which are transitive on partitions of this type because they are $5$-transitive).
\qed 
\end{pf}

%{\color{green}
%It is worth pointing out that Levi, McFadden and McAlister \cite[p.464]{lmm} proposed the problem of finding all pairs $(G,a)$, where $G\le S_n$ and $a\in T_n$, such that $$\langle G,a\rangle\setminus G = \langle g^{-1}ag\mid g \in G\rangle.$$
%The previous theorem yields the following partial answer to this question.
%
%\begin{cor}
%Let $t$ be a singular map in $T_n$, the full transformation monoid on $\Omega:=\{1,\ldots ,n\}$,  and suppose that $t$ has kernel type $(l_1,\ldots,l_k)$, with $k\ge n/2$; let $G$ be a group having only one orbit in the partitions of that type. Then 
%$$\langle G,a\rangle\setminus G = \langle g^{-1}ag\mid g \in G\rangle.$$
%\end{cor}
%}

\section{Groups with only one orbit on the image}\label{sem2}

We turn now to semigroups $\langle t,G\rangle$, where $G$ is transitive on
the image of $t$ (that is, $G$ is $(n-k)$-homogeneous, where $k\ge n/2$ is the
rank of $t$). 
%As can be easily anticipated we cannot say about the semigroups containing one of these groups as much as we said about the semigroups considered in the previous section.

\begin{thm}\label{lastmain}
Let $G$ be a primitive group with just one orbit on $(n-k)$-sets, where
$1\le k\le n/2$. Let $t\in T_n$ be a map of rank $n-k$. Then
\begin{enumerate}
\item\label{i1c} $\langle G,t\rangle\setminus G$ and $\langle g^{-1}tg \mid g\in G\rangle$ have the same idempotents;
%\item $\Aut(\langle G,t\rangle)\cong\{f\in S_n:f^{-1}\langle G,a\rangle f
%=\langle G,a\rangle\}$.
\item\label{i2c} $\Aut(\langle G,t\rangle)\cong N_{S_n}(\langle G,t\rangle)$. 
%{If $k>1$, then we already know the list of the groups and their normalizers; for $k=1$ the situation is much more complex.}
\item\label{i2cb)} For $k\ge3$, the list of $3$-homogeneous groups that satisfy $$\Aut(\langle G,t\rangle)\cong N_{S_n}(G)$$ is the following: 
\begin{itemize}\itemsep0pt
\item $G=N_{S_n}(G)$, that is, 
\begin{itemize}\itemsep0pt
\item[(i)] $S_n$.
\item[(ii)] $\pgaml(2,q)$ for $k=3$.
\item[(iii)] $\agl(d,2)$ for $k=3$.
\item[(iv)] $\agaml(1,8)$,
$M_{11}$ ($k=4$), $M_{11}$ (degree~$12$, $k=3$), $M_{12}$ ($k=5$), $2^4:A_7$, $M_{22}:2$ ($k=3$),
$M_{23}$ ($k=4$), $M_{24}$ ($k=5$), and $\agaml(1,32)$ ($k=4$).
\end{itemize}
\item $G=A_n$;
\item $G=\agl(1,8)$, $\pgl(2,8)$, $\pgl(2,9)$, $M_{10}$, $\psl(2,11)$,
$M_{22}$, $\pxl(2,25)$, or $\pxl(2,49)$, with $k=3$, $\lambda=(4,1,\ldots,1)$.
\end{itemize}
The list is complete with the possible exception of the groups $\pxl(2,q)$ for $q\ge169$. 
\item\label{i3c} Let $A\subseteq T_n$ be a set of rank $k$ maps such that $\langle A,G\rangle$ generates all maps of rank at most $k$ and $A$ has minimum size among the subsets of $T_n$ with that property. Then the maximum sizes that  $A$ can have is 
given in Table~\ref{t:gens3}. 
\end{enumerate}
\end{thm}

{
\begin{table}[htbp]
\[\begin{array}{|c|c|c|c|}
\hline
\pbox{5cm}{\mbox{\tiny{Rank}} \\ \text{\tiny{$n-k$}} }& { |A|} &\pbox{5cm}{\   \\ {\tiny Sample $k$-homogeneous groups}\\  {\tiny  attaining the bound for $|A|$} \\ }&\pbox{5cm}{{\tiny Minimum number of} \\ {\tiny generators for a primitive} \\ {\tiny $k$-homogeneous group}  }\\
\hline\hline
\text{\tiny{$n-1$}} & \frac{(n-1)}{2} &C_p,D_{p} \ (\mbox{{\tiny $n$ odd prime}})  & \frac{C\log n}{\sqrt{\log \log n}} \\
\text{\tiny{$n-2$}} & O(n^2) &\mbox{Example \ref{expl2.1}} &2\\
\text{\tiny{$n-3$}} & O(n^3) &\mathrm{PSL}(2,q),\mathrm{P}\Gamma\mathrm{L}(2,q) &2\\
\text{\tiny{$n-4$}} & 12160 &\mathrm{P}\Gamma\mathrm{L}(2,32)  \ (\mbox{{\tiny $n=33$}}) &2\\
\text{\tiny{$n-5$}} & 77&M_{24} \ (\mbox{{\tiny $n=24$}})&2\\
\text{\tiny{$n-k$ $(k\ge5)$}} & p(k)&S_n,A_n &2 \\
\hline
\end{array}\]
\caption{\label{t:gens3} { Worst case scenario of the smallest number of rank $n-k$ maps needed to together with a $k$-homogeneous group $G$ generate all the maps of rank at most $n-k$.}}
\end{table}
}
For more precise values depending on the group chosen, see the tables in 
Section \ref{groups} of the paper.
 
\begin{proof}
McAlister \cite{mcalister} proved that for any group $G\le S_n$ and any transformation $a\in T_n$, the semigroups $\langle a,G\rangle\setminus G$ and $\langle  g^{-1}ag\mid g\in G\rangle$ have the same idempotents. This proves (\ref{i1c}).

A transitive group $G$ is said to synchronize a map $t$ if the semigroup $\langle G,t\rangle$ contains  a constant map (and hence, by transitivity, all constant maps). It is proved that primitive groups  synchronize every singular map of rank at least $n-4$ (see \cite{ABCRS,ArCa14,rystsov}). It is also known that $2$-homogeneous groups, together with any singular map, generate all the constant maps (\cite{ArCaSt15,mcalister}). Therefore, under the assumptions of the theorem, if the primitive group $G$ has only one orbit on the $k$-sets, for $n>k\ge n/2$, then $G$ together with any rank $k$ map $t$ generates  all the constants and hence the automorphisms of $S:=\langle t,G\rangle$ are induced under conjugation by the elements in $N_{S_n}(S)$. This implies (\ref{i2c}). 

The more detailed description included in  (\ref{i2cb)}) follows from Theorem \ref{main3}. 

Regarding (\ref{i3c}), we start by observing  that all maps in $\langle G,t\rangle$ having the same rank of $t$, have also the same kernel type of $t$. Therefore, to generate all rank $k$ maps with $G$ and a set $A$ of rank $k$ maps, $A$ must contain maps whose kernels form a transversal  of the orbits of $G$ on each kernel type. This necessary condition turns out to be sufficient for $\langle G,A\rangle$ to generate all transformations of rank at most $k$. In fact,  given any $k$-partition $P$ and any transversal $S$ for $P$, there exists $p\in A$ and $g\in G$ such that $P=\ker(gp)$. In addition, since  $G$ has only one orbit on the $k$-sets, it follows that there exists $h\in G$ such that the image of $gph$ is $S$. {Therefore,  we infer that $\rank(phg)=\rank(p)$; thus, every element in $\langle phg\rangle$ has the same rank of $p$ and, for some natural number $\omega$,  we have that $(phg)^\omega$ is idempotent and the same holds for $e:=g(phg)^\omega g^{-1}$. In addition, $\ker(e)=P$ and the image of $e$ coincides with the image of $ph$ which is $S$.  Since $P$ and $S$ were arbitrary, it follows that $\langle A,G\rangle$ contains all rank $k$ idempotents of $T_n$. It is well known (\cite{ar2000}) that the rank $k$ idempotents generate all maps of rank at most $k$ and hence the result follows.} 
\end{proof}

\section{On normalizers of $2$-homogeneous groups}
\label{norm}

{By the main theorems of the two previous sections, to compute the automorphisms of $\langle G,t\rangle$ (with $G$ and $t$ under the assumptions of the theorems) it is necessary to know the normalizer of $\langle G,t\rangle$ in $S_n$, which is contained in the normalizer of $G$. Therefore we provide here the normalizers of $2$-homogeneous groups.} 
 
According to a theorem of Burnside, a $2$-transitive group $G$ has a unique
minimal normal subgroup $T$, which is either elementary abelian or 
simple non-abelian. (If $G$ is $2$-homogeneous but not $2$-transitive, it
also has a unique minimal normal subgroup, which is elementary abelian.)
Thus $N_{S_n}(G)\le N_{S_n}(T)$. So, to describe the normalizers of the 
$2$-homogeneous groups $G$, we only need to look within the group
$N_{S_n}(T)/T$. Table~\ref{t:norms} gives the structure of this quotient in the
case when $T$ is simple. In the table, $G(r,s,p)$ denotes the group
$\langle a,b\mid a^r=b^s=1,b^{-1}ab=a^p\rangle$. {In all rows of the table except the second and fourth, $N_{S_n}(G)=N_{S_n}(T)$. In the second and fourth rows, we have $N_{S_n}(G)/T \cong N_{N(T)/T}(G/T)$, and this quotient is computed in the metacyclic group $G(r,s,p)$.}

\begin{table}[htbp]
\begin{center}
\begin{tabular}{|c|c|c|c|}
\hline
$T$ & Degree & $N(T)/T$ & Condition \\
\hline\hline\small
$A_n$ & $n$ & $C_2$ & \\
$\mathrm{PSL}(d,q)$ & $(q^d-1)/(q-1)$ &
$G(r,s,p)$
& $q=p^s$, $p$ prime,\\
& & & $r=\gcd(q-1,d)$ \\
$\mathrm{Sp}(2d,2)$ & $2^{2d-1}\pm 2^{d-1}$ & 1 & \\
$\mathrm{PSU}(3,q)$ & $q^3-1$ &
$G(r,s,p)$
& $q=p^s$, $p$ prime,\\
& & &  $r=\gcd(q+1,3)$ \\
$\mathrm{Sz}(q)$ & $q^2+1$ & $C_{2e+1}$ & $q=2^{2e+1}$ \\
$R_1(q)$ & $q^3+1$ & $C_{2e+1}$ & $q=3^{2e+1}$ \\
$M_{11}$ & $11$ & $1$ & \\
$M_{11}$ & $12$ & $1$ & \\
$M_{12}$ & $12$ & $1$ & \\
$A_7$ & $15$ & $1$ & \\
$M_{22}$ & $22$ & $C_2$ & \\
$M_{23}$ & $23$ & $1$ & \\
$M_{24}$ & $24$ & $1$ & \\
$\mathrm{HS}$ & $176$ & $1$ & \\
$\mathrm{Co}_3$ & $276$ & $1$ & \\
\hline
\end{tabular}
\end{center}
\caption{\label{t:norms}Normalizers of almost simple $2$-transitive groups}
\end{table}

Note that there are a few small exceptions: $\mathrm{PSL}(2,2)$, 
$\mathrm{PSL}(2,3)$, $\mathrm{PSU}(3,2)$, $\mathrm{Sz}(2)$, $\mathrm{Sp}(4,2)$,
and $R_1(3)$ are not simple. The first four of these are solvable; the fifth
has a simple subgroup of index $2$ isomorphic to $A_6$; and the last
has a simple subgroup of index~$3$ isomorphic to $\mathrm{PSL}(2,8)$.

Now we consider the $2$-homogeneous affine groups.
In each case, the classification gives a subgroup $H$ (not necessarily
$2$-homogeneous) which must be contained in $G$. The group $H$ contains
the translation group $T$ of $G$, so $G=TG_0$ and $H=TH_0$. Thus, as in the
other case, we have $N_{S_n}(G)\le N_{S_n}(H)$, so again we have to compute
the normalizer within the group $N_{S_n}(H)/H\cong N_{S_{n-1}}(H_0)/H_0$.
Table \ref{t:aff_norm} gives the structure of this quotient group. The groups
$G(r,s,p)$ are the same as defined earlier. In all cases not shown in the
table, the quotient is abelian, and so the normalizers of $H$ and $T$ coincide,
and we have not listed them explicitly.

\begin{table}[htbp]
\begin{center}
\begin{tabular}{|c|c|c|c|}
\hline
Degree & $H_0$ & $N(H_0)/H_0$ & Condition \\
\hline\hline
$q^n$ & $\mathrm{SL}(n,q)$ & $G(r,s,p)$ & $q=p^s$, $r=q-1$ \\
$q$ & $C_{(q-1)/2}$ & $C_{2s}$ & $q=p^s$ odd \\
$q^{2n}$ & $\mathrm{Sp}(2n,q)$ & $G(r,s,p)$ & $q=p^s$, $r=q-1$ \\
\hline
\end{tabular}
\end{center}
\caption{\label{t:aff_norm}Normalizers of affine groups}
\end{table}

We have not attempted to make a similar classification of normalizers of
primitive groups, since this problem is as difficult as finding normalizers
of arbitrary transitive groups, as the following example shows.

Let $m\ge3$, and let $K$ be a transitive group of degree $k$. Let
$G$ be the wreath product $S_m\wr K$ in its power action of degree $m^k$.
Then $G$ is primitive, and its normalizer in the symmetric group of degree
$m^k$ is $S_m\wr N_{S_k}(K)$.
 
 \section{Problems}
If the following question has an affirmative answer (as we conjecture), then the list in Theorem \ref{main3} is complete. 
 \begin{problem}
 Is it true that for $G=\pxl(2,q)$, $q\ge 169$ and $\lambda=(4,1,\ldots,1)$, no pair $(G,\lambda)$ is closed?  
 \end{problem}

The next problem looks within reachable boundaries. 
 
 \begin{problem}
 Prove for $2$-homogeneous groups an analogous of Theorem \ref{main3}.
 \end{problem}
 
Unlike the previous, the next problem is certainly extremely difficult. 

 \begin{problem}
 Prove for primitive groups an analogous of Theorem \ref{main3}.
 \end{problem}
 
The next problem was introduced in Section \ref{orbs and norms}, but only some remarks were then done. A full solution is still out there.
 \begin{problem}
Given an orbit of the
$k$-homogeneous group $G$ on $(n-k)$-partitions,
what is the subgroup of the normalizer of $G$, in $S_n$, which
fixes that orbit?
 \end{problem}

The results on normalizers of $2$-homogeneous groups suggest the following generalization.

\begin{problem}
Let $G$ be a family of primitive groups that has been classified (e.g., primitive groups of rank $3$). Build for the groups in that family a table similar to Table \ref{t:norms}, describing the normalizers of these groups in the symmetric groups of the same degree.
\end{problem}

In order to give a sharper version of Theorem \ref{lastmain} (2), it would be useful to classify the primitive groups having a $2$-homogeneous normalizer in $S_n$. 
\begin{problem}
Classify the primitive groups $G\le S_n$ such that $N_{S_n}(G)$ is $2$-homogeneous. 
\end{problem}

  The next conjecture seems a little wild, but might be true. 
  
  \begin{problem}
  Let $G$ be a primitive group and $t\in T_n\setminus S_n$. Then all the automorphisms of $\langle G,t\rangle$ are induced (under conjugation) by the elements in $N_{S_n}(\langle G,t\rangle)$.
  \end{problem}
  
  If $t$ has rank, at least, $n-4$, then we know (by \cite{ABCRS}) that $\langle G,t\rangle$ contains all the constant maps and hence (by \cite{Sullivan}) the conjecture above holds. Similarly, if $t$ has rank at most $2$, we know by \cite{neu} that   again $\langle G,t\rangle$ contains all the constant maps and the conjecture holds as above. Also by \cite{neu} we know that there are primitive groups $G$ and singular maps $t$ such that $\langle G,t\rangle$ do not contain the constant maps. The question is: what happens with the automorphisms of those semigroups? We conjecture they are all induced by the elements of the normalizer (in $S_n$) of the semigroup, but have no proof. 

{
As shown above, every $2$-homogeneous group is $2$-generated. However the GAP library of $2$-transitive groups contains default sets of generators that in the major part of the cases have size larger than $2$.
\begin{problem} Produce a library of generating sets of size $2$ for all the degree $n$ and  $k$-homogeneous  primitive groups in the GAP library (for $k\ge 2$). 
\end{problem}
}

{
Slightly connected to the previous problem is the following.
\begin{problem} Include in GAP a very effective function to find the homogeneity of  a given permutation group. 
\end{problem}
}

{
The next problem deals again with GAP libraries.
\begin{problem}
Let $G$ be a $k$-homogeneous degree $n$ primitive group in the GAP library of primitive groups. 
Produce a minimal set $A$ of degree $n$ transformations of rank $k$ such that $\langle G,A\rangle$ generates all the transformation of rank at most $k$. 
\end{problem}
}

{
In \cite{Le96} it is proposed the problem of finding the groups that can be the normalizers in $S_n$ of some semigroup $S\subseteq T_n$. The main theorems of this paper provide some answers for that question, but we would like to propose the following conjecture.
\begin{problem}		
Is it true that a group $G$ is the normalizer in $S_n$ of a semigroup $S\le T_n$ if and only if $G$ is the normalizer  in $S_n$ of some group $H\le S_n$? 
\end{problem}			
}

{%\color{blue}
Our final problem asks for a sharper version of Theorem \ref{lastmain}, (2). 
\begin{problem}
For every pair $(G,\lambda)$, where  $G\le S_n$ is a $k$-homogeneous group and  $\lambda$ is an $(n-k)$-partition of $n$, classify the groups $N_{S_n}(\langle G,t\rangle)$, where $t$ is any map whose kernel has type $\lambda$. 
\end{problem}
}
 
 {\bf Acknowledgement.} 
This work was developed within FCT project CEMAT-CI\^{E}NCIAS (UID/Multi/04621/2013).

\end{document}